\documentclass[11pt]{amsart}
\usepackage[latin1]{inputenc}
\usepackage{amsmath}
\usepackage{amsfonts}
\usepackage{amssymb}
\usepackage{graphicx}
\usepackage{fourier}
\usepackage{mathrsfs}
\usepackage{hyperref}
\usepackage{enumerate}
\usepackage{esint}
\usepackage{bm}
\usepackage{xcolor} 
\usepackage{verbatim}        
\usepackage{bbm}

\newtheorem{theorem}{Theorem}[section]

\theoremstyle{definition}
\newtheorem{definition}[theorem]{Definition}

\theoremstyle{remark}
\newtheorem{remark}[theorem]{Remark}

\newcommand{\inner}[1]{\left\langle#1\right\rangle}
\newcommand{\norm}[1]{\left\lVert#1\right\rVert}
\newcommand{\abs}[1]{\left\lvert#1\right\rvert}
\newcommand{\pa}[1]{\left( #1 \right)}
\newcommand{\rpa}[1]{\left[ #1 \right]}
\newcommand{\br}[1]{\left\lbrace #1\right\rbrace}
\newcommand{\R}{\mathbb{R}}

\newcommand{\N}{\mathbb{N}}
\newcommand{\sph}{\mathbb{S}}

\numberwithin{equation}{section}

\def\1{\mathbbm{1}}
\newcommand{\kac}{\mathbb{S}^{N-1}\pa{\sqrt{N}}}
\newcommand{\LN}{\mathcal{L}_{N}}

\newcommand{\HH}{\mathcal{H}}
\newcommand{\MM}{\mathcal{M}}
\newcommand{\EE}{\mathcal{E}}
\newcommand{\DD}{\mathcal{D}}
\newcommand{\dsn}{d\sigma_N}
\newcommand{\ZZ}{\mathcal{Z}}
\newcommand{\ZN}{\mathcal{Z}_N\pa{f,\sqrt{N}}}

\newcommand{\amit}{\textcolor{black}} 

\DeclareFontFamily{T1}{calligra}{}
\DeclareFontShape{T1}{calligra}{m}{n}{<->s*[2]callig15}{}
\DeclareMathAlphabet\mathcalligra   {T1}{calligra} {m} {n}

\title{The Entropic journey of Kac's Model}
\author{Amit Einav}

\address{Durham University, School of Mathematical Sciences, Upper Mountjoy Campus, Stockton Road, DH1 3LE, Durham, United Kingdom}
\email{amit.einav@durham.ac.uk}

\begin{document}

\dedicatory{In loving memory of Maria Concei\c{c}\~{a}o Carvalho (S\~{a}o), who is always with us in our memories and our hearts. }

\begin{abstract}
The goal of this paper is to review the advances that were made during the last few decades in the study of the entropy, and in particular the entropy method, for Kac's many particle system.  

\end{abstract}

	\maketitle
	
	\tableofcontents

\section{Introduction}\label{sec:intro}

Our fascination with a mathematical description for real life phenomena can be viewed as a cornerstone in the birth of the mathematical subject known as \textit{(modern) Analysis} back in the 17th century. Our desire to understand such phenomena has not lessened over the course of the centuries - if anything, our expanding familiarity with the universe and its mysteries only deepened it. Particular interest, especially in the last century, is given to phenomena which involve many elements. 

Systems that revolve around many particles, substances, animals or agents are all around us - from the air we breath to the galaxies around us, herds of buffaloes roaming the African plains and red blood cells coursing through our bodies. Even our attempts to navigate the crowd in a busy shopping centre is another example of a system of many elements. As common as such systems are, their mathematical investigation is far from simple.

The goal of our paper is to give a short overview on how we mathematically treat such systems and to focus our attention on one particular important model: Kac's many particle model. Kac's model, introduced in 1956 by Mark Kac (see \cite{Kac1956}), was conceived to give a probabilistic justification to the famous Boltzmann equation\footnote{as opposed to the rigorous derivation of it from Newtonian mechanics, which is known as Hilbert's 6th problem.} and is one of the first examples for the \textit{mean field limit} procedure. 

Our paper will almost solely concern itself with the \textit{functional} study of this model and its convergence to equilibrium via the so-called \textit{entropy method}\footnote{we will be remiss if we won't mention that there are many other studies on the convergence to equilibrium, such as the impressive \cite{MM2013}.}. Many contribution to this study, as well as to the Boltzmann equation which is intimately connected to Kac's model, have been made over the last 70 years. For the sake of brevity we will only mention works that are connected to the topic we present here, and even that will be mostly limited to those papers that are directly related to the results we will present. 

Further information about the history and study of Kac's model and the Boltzmann equation can be found in \cite{CCL2011review,MM2013,V2002review} and references within. 

\subsection{Many elements systems - the Microscopic, Macroscopic and Mesoscopic framework}\label{subsec:micro_macro_meso}
When one considers systems of many elements there are usually three frameworks one can adopt: \textit{microscopic, macroscopic and mesoscopic}.
\subsubsection*{Microscopic approach} This approach is, in a sense, the oldest and the most accurate approach. In the microscopic framework we view each element in the system as an individual and consider, and attempt to solve, the phase space trajectorial equations associated to these elements. In the common case where the interaction between the elements of the system is binary, i.e. happens only between two elements, a typical microscopic description of a system is given by  
\begin{equation} \label{eq:micro}
	\begin{split}
		&\frac{d}{dt}x_i(t)=v_i(t),\\
		&\frac{d}{dt}v_i(t)=\frac{1}{N}\sum_{j\not=i}K\pa{x_i(t)-x_j(t)},		
	\end{split}		
\end{equation}

where the function $K$ represents the force of the interaction between the $i-$th and $j-$th elements and is assumed to depend only on the relative distance between them. The factor of $1/N$ appearing next to the sum of the second equation (corresponding to Newton's second law) makes sure that the force acting on each element is of order of magnitude $1$.  

While extremely accurate in theory, solving a system such as \eqref{eq:micro} is not really feasible in most cases - even when we have theorems that assure us that there exists a solution. Our significant computational power is not enough to solve even relatively simple systems of this form. However, in most cases we don't really need to.

\subsubsection*{Macroscopic approach} Faced with the seemingly herculean task of trying to solve microscopic systems of equations, people have realised that for all intent and purposes we do \textit{not} need to understand the individual behaviour of each element of the system in order to understand its behaviour. Indeed, as most early systems that were investigated revolved around particles of gases or fluids, systems that comprise of elements that are not perceptible to the human eye, it was understood that the behaviour we are looking for is more \textit{statistical}, or \textit{probabilistic}, in nature. The core idea of the macroscopic approach is to consider a part of a system as a ``unit'' whose evolution we would like to investigate. This part of the system should be sufficiently large to include enough elements to facilitate probabilistic considerations, yet sufficiently small with respect to the entire system to be considered as ``point mass''. Macroscopic approaches are the realm of \textit{Fluid Dynamics} and typical equations in it include the Euler equation
\begin{equation}\label{eq:euler}
	\begin{aligned}
		&\frac{\partial u}{\partial t}+\pa{u\cdot\nabla}u=-\frac{1}{\rho}\nabla p +F,
	\end{aligned}
\end{equation}
and the Navier-Stokes equation\footnote{equation \eqref{eq:navier_stokes} represents the Navier-Stokes equation without external force.}
\begin{equation}\label{eq:navier_stokes}
	\begin{aligned}
		&\frac{\partial u}{\partial t}+\pa{u\cdot\nabla}u=-\frac{1}{\rho}\nabla p +\nu \Delta u,
	\end{aligned}
\end{equation}
where $u$ is the velocity field of the fluid, $\rho$ is the density of it, $p$ is the pressure, $F$ is the external force, and $\nu$ is the viscosity of the fluid. 

\subsubsection*{Mesoscopic approach} The mesoscopic approach appeared around the end of the 19th century in the mathematical study of the kinetic theory of gasses. It was spearheaded by figures such as Maxwell and Boltzmann. This approach is an ``in between'' approach between the microscopic and macroscopic. The object of investigation of mesoscopic equations is an \textit{average element} of the system. Consequently, mesoscopic equations usually consider the evolution of probability density functions. 

An extremely important (and one of the first) mesoscopic equation is the so-called Boltzmann equation, given by
\begin{equation}\label{eq:boltzmann}
	\partial_t f(t,x,v)+v\cdot \nabla_x f(t,x,v)=Q_B(f,f)(t,x,v),
\end{equation} 
with
\begin{equation}\nonumber
	\begin{split}
		Q_B&(f,f)(x,v)=\int_{\R^d \times \mathbb{S}^{d-1}}B\pa{v,v_\ast, \sigma}\\
		&\pa{f\pa{x,v^\prime}f\pa{x,v_\ast^\prime}-f\pa{x,v}f\pa{x,v_\ast}}dv_\ast d\sigma,
	\end{split}
\end{equation}
where 
\begin{equation}		\label{eq:scattering}
	\begin{split}
		&v^\prime=\frac{v+v_\ast}{2}+\frac{\abs{v-v_\ast}\sigma}{2},\\
		&v_\ast^\prime=\frac{v+v_\ast}{2}-\frac{\abs{v-v_\ast}\sigma}{2},
	\end{split}
\end{equation}
and $B\pa{v,v_\ast, \sigma}$ being the collision kernel (more on which shortly).

The Boltzmann equation describes the evolution of an average particle in dilute (or rarefied) gas with no external forces which is governed by two processes:
\begin{enumerate}[1. ]
	\item The left hand side of \eqref{eq:boltzmann}, known as the \textit{transport part} of the equation, represents the free motion of the average particle prior to its collision with another particle. It is based on the simple transport equation
	\begin{equation}\nonumber
		\begin{aligned}
			&	\partial_t f(t,x,v)+v\cdot \nabla_x f(t,x,v)=0,\\
			&f\pa{0,x,v}=f_0(x,v),		
		\end{aligned}
	\end{equation}
	whose solution 
	$$f(t,x,v)=f_0\pa{x-vt,v}$$
	is constant on trajectories of paths of constant velocities. The transport part is very common in many physically relevant mesoscopic equations. 
	\item The right hand side of \eqref{eq:boltzmann}, known as the \textit{collision or interaction} part, represents the effect of the interactions, collisions in the case of gases, on the evolution of the probability density. In the case of the Boltzmann equation, and many other mesoscopic equations, these collisions happen when two (average) particles reach the same spatial position. How will this affect our ``possibility'' to find a particle at time $t$ in this position $x$ and velocity $v$? We need to consider two scenarios:
	\begin{itemize}
		\item[-] The particles arrived at velocities $v^\prime$ and $v_\ast^\prime$ to the position $x$ and collided in a way that preserves the energy and momentum. Consequently, they have scattered in accordance to a unit vector $\sigma$ and ended up with velocities $v$ and $v_\ast$, thus \textit{increasing} the possibility to find an average particle at position $x$ and velocity $v$. The connection between $v^\prime$, $v_{\ast}^\prime$, $v$, $v_\ast$, and the scattering angle $\sigma$ is given by \eqref{eq:scattering}\footnote{one can show that this relation is the only viable one if one assumes conservation of energy and momentum.} and this process is encoded in the \textit{gain} term of $Q_B$
		$$\int_{\R^d\times \sph^{d-1}}B\pa{v,v_\ast, \sigma} f\pa{x,v^\prime}f\pa{x,v_\ast^\prime}dv_\ast d\sigma.$$
		The function $B\pa{v,v_\ast, \sigma}$, the \textit{collision kernel}, gives us the intensity of such collision and depends only on the physics of the problem. The integration over $\sigma$ and $v_\ast$ is present to take into account all possible scattering and resulting additional (less relevant) velocity. There is, however, some subtle but fundamental assumption here: the collision process we've described speaks about two (average) particles being at the same position $x$ with velocities $v^\prime$ and $v_\ast^\prime$ respectively. This information is encoded in the function $f_2\pa{x,x,v^\prime,v_{\ast}^\prime}$ where $f_2$ is the \textit{joint probability density} of these average particles. The gain term, however, contains the expression $f\pa{x,v^\prime}f\pa{x,v_\ast^\prime}$ instead. This implies an assumption of \textit{independence} between these particles prior to the collision. This is known as Boltzmann's \textit{molecular chaos} or \textit{Stosszahlansatz}.
		\item[-] The same consideration as above, together with Boltzmann's molecular chaos assumption, gives us the \textit{loss} term in $Q_B$ 
		$$-\int_{\R^d\times \sph^{d-1}}B\pa{v,v_\ast, \sigma} f\pa{x,v}f\pa{x,v_\ast}dv_\ast d\sigma$$
		which expresses the \textit{decrease} in the possibility to find an average particle at position $x$ and velocity $v$ and corresponds to the process that the two particles arrived to the point $x$ at velocities $v$ and $v_\ast$ and their resulting post collision velocities do not include (up to a set of measure zero) the velocity $v$.
	\end{itemize}
\end{enumerate}
It is worth to mention that in most situations the collision kernel $B$ is given by
\begin{equation}\label{eq:collision_kernel}
	B\pa{v,v_\ast,\sigma}=\abs{v-v_\ast}^\gamma b\pa{\cos\theta}.
\end{equation}
with $\gamma$, known as the \textit{hardness of the potential}, being in  $[-d,1]$.

Nowadays we consider many equations to be ``kinetic equations'', even when they do not have a clear mesoscopic interpretation, and use common tools to investigate them. Examples include Fokker-Planck equations, coagulation and fragmentation equations, and reversible and irreversible reaction-diffusion equations. 

\subsection{Moving between the Microscopic, Macroscopic, and Mesoscopic frameworks}\label{subsec:hydro_and_mean_field}
It is evident from the description in \S\ref{subsec:micro_macro_meso} that the different frameworks one has for equations and models that pertain to many element phenomena come in hierarchical order. It shouldn't come as a big surprise then that we can try and ``move'' between these approaches. The process of inferring an equation in one framework from another framework requires some limiting procedure. The two common procedures we employ are expressed in the following diagram:
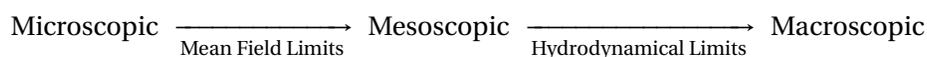
\vspace*{0.1cm}
$$\text{Microscopic}\;\;\xrightarrow[\text{Mean Field Limits}]{}\;\;\text{Mesoscopic}\;\;\xrightarrow[\text{Hydrodynamical Limits}]{}\;\;\text{Macroscopic}$$
\vspace*{0.1cm}
We will mostly focus on mean field limits, but would mention hydrodynamical limits for the sake of completion. 
\subsubsection*{From meso to macro - Hydrodynamical limits}  
The mesoscopic framework ``lives'' in the microscopic world of individual elements yet concerns itself only with average ``individual'' behaviours. These behaviours, we imagine, are enough for us to understand how the macroscopic ``unit'' evolves. In order for us to be able to achieve this goal, however, we must scale our spatial and time variables since these ``units'' of the macroscopic world contain a large quantity of mesoscopic elements, and the time scale where phenomena occur in these levels is not of the same order of magnitude.

This scaling is usually done by introducing a new ``smallness'' parameter, pertaining to the ``zooming out'' in space and time and the impact of the scaling on the strength of interactions between the elements. One can use such approach, which known as \textit{hydrodynamical limits}, to find connection between fluid equations, such as the Euler equation, and the Boltzmann equation. As this topic is outside the remit of this review paper, we will not elaborate more on it.

\subsubsection*{From micro to meso - Mean field limits}  
It shouldn't come as a shock that we should expect to be able to get information on the behaviour of an average element of a system, i.e. on the mesoscopic framework, from the complete deterministic behaviour of every element in it, i.e. the microscopic framework. 

The first step one must undertake to achieve such a goal is to recast the trajectorial microscopic equations, expressed in \eqref{eq:micro}, in a more probabilistic way. Let us, thus, consider a system of $N$ elements, each of which lives in the phase space $\mathcal{X}\times \mathcal{V}$. In order to consider the problem from a probabilistic perspective we need to explore the \textit{probability density function} of the entire ensemble, which we will denote by $F_N\pa{\bm{X}_N,\bm{V}_N}$ with $\pa{\bm{X}_N,\bm{V}_N}\in \pa{\mathcal{X}\times \mathcal{V}}^N$. As $F_N$ remains constants on the trajectories of the individual elements of the system we find that the system of equations \eqref{eq:micro} is equivalent, at least formally, to the following equation
\begin{equation}\label{eq:master}
	\begin{split}
		\partial_t F_N\pa{\bm{X}_N,\bm{V}_N}&+\bm{V}_N \cdot \nabla_{\bm{X}_N}F_N\pa{\bm{X}_N,\bm{V}_N} \\
		&+ \frac{1}{N}\sum_{i=1}^N\sum_{i\not=j}K\pa{x_i-x_j} \cdot \nabla_{v_i} F_N\pa{\bm{X}_N,\bm{V}_N}=0
	\end{split}
\end{equation}
which is known as the \textit{master equation} or the \textit{Liouville equation}. 

As the mesoscopic approach concerns itself with only one average element, we need to reduce \eqref{eq:master} to an equation that relates to such element. Before we do so we notice that in order for the above to make sense, and indeed in order for us to even be able to consider a probabilistic approach here, we shouldn't be able to distinguish between the elements of the system. Functionally, what this means is that the probability density function must be \textit{symmetric} under permutation of its phase-space variables, i.e. for any $\sigma$ the group of permutation of $\br{1,\dots, N}$, we must have that 
$$F_N\pa{x_1,\dots, x_N,v_1,\dots,v_N} = F_N\pa{x_{\sigma\pa{1}},\dots, x_{\sigma\pa{N}},v_{\sigma\pa{1}},\dots, v_{\sigma\pa{N}}}.$$
The evolution of one average element of the system, therefore, is determined by an equation for the \textit{first marginal} of $F_N$ which we will denote by $F_{N,1}$. Since $F_N$ is symmetric, it doesn't matter which variable we consider. To find the evolution of this marginal we need to integrate out the remaining variables in \eqref{eq:master}. This, however, will not yield a closed equation since the term that involves the function $K\pa{x_i-x_j}$ causes correlations that can't be integrated out. Indeed, using our symmetry assumption we find that the evolution of $F_{N,1}$ is given by
\begin{equation}\label{eq:first_marginal_with_second}
	\begin{aligned}
		&\partial_t F_{N,1}\pa{x_1,v_1}+v_1 \cdot \nabla_{x_1}F_{N,1}\pa{x_1,v_1} \\
		&+ \frac{N-1}{N}\int_{\mathcal{X} \times \mathcal{V}}K\pa{x_1-x_2} \cdot \nabla_{v_1} F_{N,2}\pa{x_1,x_2,v_1,v_2}dx_2dv_2=0.
	\end{aligned}
\end{equation}
An attempt to find the evolution of the second marginal, $F_{N,2}$, in order to close the above equation will result in an equation that will depend (via an integration just like above) on $F_{N,3}$. In general one can form a hierarchy of $N$ equations, known as the \textit{BBGKY}\footnote{Bogoliubov-Born-Green-Kirkwood-Yvon.} hierarchy, which, up to the $N-$th level, details the evolution of $F_{N,k}$ with respect to $F_{N,k+1}$.

As can be seen from the above, we can't expect to be able to close our hierarchy as it stands. At this point there are two crucial observation we would like to make:
\begin{enumerate}[1)]
	\item In order for us to be able to justify considering an average element of the system we must have an enormous amount of elements in it. Mathematically speaking, this means that we would like to consider the limit $N\to\infty$. 
	\item Most system that we consider in our investigations do not evolve ``randomly'' (for lack of a better word) and we expect some phenomenon to emergence, a sort of \textit{asymptotic correlation} between the elements that becomes more and more pronounced as the number of elements in the system increases. One example, which came from the study of dilute gases, is that since in such gases particles hardly see each other, one could expect that any \textit{fixed} group of $k$ particles will become more and more \textit{independent} as the number of gas particle increases. In other words, 
	\begin{equation}\label{eq:intuitive_chaos}
		\begin{split}
			&F_{N,1}(x_1,v_1)\underset{N\text{ large}}{\approx} f(x_1,v_1),\\
			&F_{N,2}(x_1,x_2,v_1,v_2)\underset{N\text{ large}}{\approx} f(x_1,v_1)f(x_2,v_2), \\
			&\vdots\\
			&F_{N,k}(\bm{X}_k,\bm{V}_k)\underset{N\text{ large}}{\approx} f^{\otimes k}\pa{\bm{X}_k,\bm{V}_k}.
		\end{split}
	\end{equation}
\end{enumerate}
The condition expressed in \eqref{eq:intuitive_chaos}, which we will get back to in the next section, is known as \textit{chaos} or \textit{chaoticity}. Assuming that this condition holds under the evolution of the master equation we find that taking the limit of the number of elements to infinity in \eqref{eq:first_marginal_with_second} gives the equation
\begin{equation}\nonumber
	\begin{aligned}
		\partial_t f\pa{x,v}+v \cdot \nabla_{x}f(x,v) 
		+  \pa{K \ast \rho} (x)  \nabla_v f(x,v)=0,
	\end{aligned}
\end{equation}
where $\rho(x)=\int_{\mathcal{V}}f(x,v)dv$. The above represents the equation of a \textit{limiting} average element.

The process of using asymptotic correlation on a marginal equation to attain a limit equation is known as the \textit{mean field limit} procedure. The solution to the limiting equation is known as the \textit{mean field limit}.
There are a few things we should emphasise at this point:
\begin{itemize}
	\item The true starting point of the mean field limit procedure was not the microscopic framework and the set of equations \eqref{eq:micro} - it was the master equation \eqref{eq:master}. This observation allows us to start our mean field study \textit{not} from a deterministic trajectorial system but from so-called average model, expressed via a master equation for a probability density of an ensemble of elements (which is usually governed by a Poisson jump processes). This line of investigation and the use of mean field limit approaches is an extremely active field and includes models that range from dilute gases to swarming of animals and decision making. 
	\item Asymptotic correlations are \textit{functional properties} that involve no evolution. As such, to properly apply a mean field limit approach one \textit{must} show that the appropriate asymptotic correlation \textit{propagates} with the evolution of the master equation.
	\item To this day there is only \textit{one} asymptotic correlation we use in our mean field approach - chaoticity. The notion of chaos in the form presented here was conceived by Kac in his 1956 work on his particle model (see \cite{Kac1956}).  
	
\end{itemize}
We have now covered the background and framework for the rest of our work - the study of Kac's original particle model.

\section{Kac's model and the study of the convergence to equilibrium}\label{sec:kac_model}

\subsection{Kac's model}\label{subsec:kac_intro} Showing the validity of the Boltzmann equation from Newtonian mechanics is an old problem which was formally articulated in Hilbert's famous symposium at 1900. It is commonly known as \textit{Hilbert's 6th problem}. While some progress has been made, most notable by Lanford in 1975 (for some history and a  modern take on Landord's ideas see \cite{GSRT2013}), the problem remains open. In his 1956 work, \cite{Kac1956}, Kac introduced an average model (in the sense presented in the end of the previous section) to try and give a probabilistic justification to the Boltzmann equation. 

Kac's model considered $N$ gas particles which interact via a binary collision process arriving in a Poisson stream. The jump process associated to these collisions replaces the trajectorial role of the spatial variable in the original equation and consequently Kac's model is spatially homogenous. To simplify things further, Kac assumed that each particle has a one dimensional velocity. The price of this simplification is the fact that unlike the Boltzmann equation, the collisional process in Kac's model does not preserve both energy and momentum\footnote{Kac's model can be (and have been) extended to a higher dimensional model where this has been rectified (see \cite{McKean1967})}.

The process that governs the evolution of Kac's model is the following: when the Poisson clock chimes, two particles who are chosen uniformly from the ensemble, say the $i-$the and $j-$th particle, collide and scatter with a random scattering angle, $\theta$, which is distributed uniformly on $[0,2\pi]$. The resulting new velocities of the particles are given by
\begin{equation}\label{eq:collision_velocitis}
	\begin{split}
		&v_i\pa{\theta}=v_{i}\cos\pa{\theta}+v_{j}\sin\pa{\theta},\\
		&v_j\pa{\theta}=-v_{i}\sin\pa{\theta}+v_{j}\cos\pa{\theta}.
	\end{split}
\end{equation}
Kac's evolution, i.e. his master equation, is constructed in a similar way to Boltzmann equation and includes a gain and loss terms. It is given by
\begin{equation}\label{eq:master_kac}
	\partial_t F_N \pa{t,\bm{V}_N}=\mathcal{L}_N F_N = N(\mathcal{Q}-I)F_N\left(t,\bm{V}_N \right), \quad \bm{V}_N\in \mathbb{S}^{N-1}\pa{\sqrt{N}}
\end{equation}
where
$$\mathcal{Q}F\pa{\bm{V}_N}=\frac{1}{\left(\begin{tabular}{ c}
		$N$ \\
		$2$
	\end{tabular}\right)}\sum_{i<j}\frac{1}{2\pi} \int_{0}^{2\pi}F_N\left(R_{i,j,\theta}\pa{\bm{V}_N}\right)d\theta,$$
with 
\begin{equation}\nonumber
	\pa{R_{i,j,\theta}\pa{\bm{V}_N}}_{l}=\begin{cases}
		v_l & l\not=i,j,\\
		v_i\pa{\theta} & l=i,\\
		v_j\pa{\theta} & l=j.
	\end{cases}
\end{equation}

A couple of remarks:
\begin{itemize}
	\item The factor $N$, appearing before the gain and loss term (represented by $\mathcal{Q}$ and $I$), is there to make sure that the average time between two collisions is of order of magnitude $1$ and \textit{is independent in $N$}. This is crucial if we want to achieve a meaningful limit when $N$ goes to infinity.
	\item The phase space in Kac's model is  $\kac$ and \textit{not} $\R^N$. The reason behind this is that the collision process described above conserves the energy of the ensemble (but not the momentum). As the particles in this model are indistinguishable we know that if one of them has kinetic energy $E$ then all of them must have the same energy. Consequently, the total energy of the ensemble is $NE$. Choosing $E$ to be $1/2$ under the assumption that the mass of all particles is $1$, gives us that $\bm{V}_N\in \kac$. The fact that we are restricting ourselves to a sphere will play an important role, and add some difficulties, in the study of the long time behaviour of the solutions to our master equation. In this work we will sometimes refer to $\kac$ as \textit{Kac's sphere}. 
\end{itemize}
\subsection*{The connection between Kac's model and the Boltzmann equation}\label{subsec:kac_MFL} Following on the ideas presented in \S\ref{subsec:hydro_and_mean_field}, one can find the first equation in the BBGKY hierarchy for Kac's master equation and attempt to perform the mean field limit procedure. A straight forward calculation (which is a bit more complicated since we are working on Kac's sphere) shows that the equation for $F_{N,1}$ is given by
\begin{equation}\nonumber
	\partial_t F_{N,1}(v) =\frac{1}{\pi} \int_0^{2\pi}\int_{\mathbb{R}} \left(F_{N,2}(v(\theta),(w(\theta))-F_{N,2}(v,w)\right)d\vartheta dw
\end{equation}
where $v\pa{\theta}$ and $w\pa{\theta}$ are given by the same formula as $v_i\pa{\theta}$ and $v_j\pa{\theta}$ in \eqref{eq:collision_velocitis}. In order to proceed with the mean filed limit approach on the above we require two things: an asymptotic correlation and a proof of its propagation.

The notion of \textit{chaos}, mentioned in \S\ref{subsec:kac_MFL}, was conceived by Kac for this purpose and was presented for the first time in the same paper, \cite{Kac1956}. The formal definition of this notion, extended to general probability measures, is given by:
\begin{definition}\label{def:chaos}
	Let $\mathcal{X}$ be a Polish space. We say that a sequence of symmetric probability measures, $\br{\mu_N}_{N\in\N}\in \mathcal{P}\pa{\mathcal{X}^N}$, is $\mu_0-$chaotic for some probability measure $\mu_0\in \mathcal{P}\pa{\mathcal{X}}$ if
	$$\Pi_k\mu_{N} \underset{N\to\infty}{\overset{\text{weak}}{\longrightarrow}}\mu_0^{\otimes k}$$
	where $\Pi_k \mu_N$ is the $k-$the marginal of $\mu_N$.\footnote{this condition can be reformulated using empirical measures. $\br{\mu_N}_{N\in\N}$ will be $\mu_0-$chaotic if and only if for any random vector $\pa{X_1,\dots,X_N}$ with law $\mu_N$ we have that the random \textit{empirical measure}
		$$\mu_N = \frac{1}{N}\sum_{i=1}^N \delta_{X_i}$$
		converges in law towards the deterministic measure $\mu_0$. } 
	
	In the case of Kac's model we say that a sequence of symmetric probability density functions on $\kac$, $\br{F_N}_{N\in\N}$, is $f-$chaotic, where $f$ is a probability density on $\R$, if $F_N d\sigma_N$ is $f dx-$chaotic, where $d\sigma_N$ is the uniform probability measure on $\kac$.
\end{definition}
Following on his definition of chaoticity, Kac has shown that this notion is propagated by his master equation and concluded his model's mean field limit equation
\begin{equation}\label{eq:boltzmann_kac}
	\partial_t f(v) = \frac{1}{\pi}\int \left(f(v(\theta))f(w(\theta))-f(v_1)f(v_2) \right)d\theta dw,
\end{equation}
which is known as the \textit{Boltzmann-Kac equation}. When comparing the above to \eqref{eq:boltzmann}, one can view it as a $1-$dimensional version of the Boltzmann equation for Maxwell molecules, i.e. for the case where $\gamma=0$ and $b$ is a constant in \eqref{eq:collision_kernel}. With that, Kac has managed to give a probibalistic justification to the validity of the Boltzmann equation. 


\subsection{Convergence to equilibrium - the spectral gap}\label{subsec:spectral_gap}
At the time of conception of Kac's model, the study of non-linear equations was not as developed as it is today. Besides giving a validation to Boltzmann equation (and establishing the mean field limit approach with his notion of chaoticity), Kac was interested in trying to use his limiting procedure to ``push down'' information from his many particle model, which is governed by a simple yet dimensional dependent operator, to its limit equation. In particular, a question that was, and is to this day, of great interest to the community was that of the rate of convergence to equilibrium for this equation. 

A closer look at the operator that governs Kac's master equation, $\mathcal{L}_N$, reveals the following:
\begin{itemize}
	\item \sloppy $\LN$ is composed of $2-$dimensional rotations and averaging. As such it is a bounded linear operator from $L^p\pa{\kac,d\sigma_N}$ to  $L^p\pa{\kac,d\sigma_N}$ for all $p\in [1,\infty]$. 
	\item Continuing on the above, a simple calculation shows that $\LN$ is also self-adjoint on $L^2\pa{\kac, d\sigma_N}$.
	\item \sloppy $\LN F =0$ if and only if $F$ is a radial function. As our underlying space is a sphere we conclude that a normalised basis to the kernel of $\LN$ in $\pa{L^2\pa{\kac},d\sigma_N}$, 
	and the only steady state to the equation with a unit mass, is $F_N\equiv 1$. The fact that we have a unique normalised steady state is another reason why one needs to restrict the ensemble to a sphere in Kac's model.
\end{itemize}

The fact that $\LN$ is self-adjoint on $L^2\pa{\kac,d\sigma_N}$ and has a one dimensional kernel implies that the convergence to the equilibrium in Kac's model, which is the probability density $F_N\equiv1$, is determined by the \textit{spectral gap} of the operator $\LN$:
$$\Delta_N=\inf_{F_N\perp 1,\; \norm{F_N}_{L^2\pa{\kac, d\sigma_N}}=1} \inner{F_N,\LN F_N}.$$
Kac has realised that in order to be able to utilise this information for his mean field limit he needed to find a lower bound on $\Delta_N$ that is \textit{uniform in $N$}. He conjectured in his work that 
$$\Delta=\inf_{N\in\N} \Delta_N >0.$$
It took 44 years to resolve this conjecture. The first proof was given by Janvresse in \cite{J2001} and later, in their work \cite{CCL2003},  Carlen, Carvalho and Loss showed that
\begin{equation}\label{eq:spectral_gap}
	\Delta_N=\frac{N+2}{2\pa{N-1}}.
\end{equation}
\amit{It is worth the note that the study of spectral gaps in many particle models, and general Markov processes that relate to the notion of chaoticity, remains a fascinating and not relevant topic. We would like to mention in particular the recent work of Carlen, Carvalho and Loss \cite{CCL2020}.}

While \amit{equality \eqref{eq:spectral_gap}} seems very encouraging, it was known long before the conjecture was resolved that the spectral gap is not the way to go forward. The reason behind this is the fact that the distance that is associated to the spectral gap, the $L^2$ norm, does not adhere to the notion of chaoticity. 

Intuitively speaking, if $\br{F_N}_{N\in\N}$ is $f-$chaotic then one could imagine that $F_N\approx f^{\otimes N}$. This is quite untrue as it neglects correlation that appear, but it epitomises the underlying problem of the $L^2$ distance as the $L^2$ norm of a tensorisation (at least on a tensorised space) becomes the \textit{product} of the $L^2$ norms of the ``generator'' of this tensorisation. On Kac's sphere one can find a sequence $\br{F_N}_{N\in\N}$ of probability densities such that 
$$\norm{F_N}_{L^2\pa{\kac,d\sigma_N}} \geq C^N$$
for some $C>1$. Combining the above with  \eqref{eq:spectral_gap} yields the sharp estimate
$$\norm{F_N(t)-1}_{L^2\pa{\kac,d\sigma_N}} \leq e^{-\frac{\pa{N+2} t}{2(N-1)}}\norm{F_N(0)-1}_{L^2\pa{\kac,d\sigma_N}}$$
showing the problem we face: as the right hand side of the above dominates an expression of the form $C^N e^{-\frac{t}{2}}$ we see that we won't be able to take $N$ to infinity and get something meaningful. Intuitively, the above shows that the time we need to wait before seeing significant decay of our solution is of order $N$, i.e. the number of particles. 

\subsection{Convergence to equilibrium - the entropy}\label{subsec:entropy}
The failure of the spectral gap method due to its underlying distance, as described above, made people realise that we need to consider a different type of distance with which we measure how far a solution is from equilibrium. It didn't take very long to realise that perhaps we should take a leaf from Boltzmann's book and consider the (non-linear) ``distance'' that is \textit{the entropy}. Modelled after the Boltzmann relative entropy\footnote{we call is a \textit{relative} entropy as it measures a function relative to another, in our case the equilibrium. There are many possibilities for entropies and we refer the reader to \cite{AMTU2001} where they can find some common examples used in other kinetic equations.}, given by
\begin{equation}\nonumber
	\mathcal{H}(f|f_\infty) = \int_{\R}h\pa{f(x)|f_\infty(x)}dx
\end{equation}
where $h\pa{x|y}=x\log \pa{\frac{x}{y}} -x +y$ and $f_\infty$ is the equilibrium of the equation, Kac's entropy is given by
\begin{equation}
	\mathcal{H}_N(F_N)=\int_{\kac}F_N \log F_N d\sigma_N.
\end{equation}
An indication that we have indeed chosen the right ``distance'' comes from its adherence to the notion of chaoticity, at least intuitively. Assuming that $F_N\approx f^{\otimes N}$ one can show that
$$\HH_N\pa{F_N} \approx \int_{\kac}f^{\otimes N}\pa{\bm{V}_N}\log\pa{f^{\otimes N}\pa{\bm{V}_N}}d\sigma_N = \sum_{i=1}^N\int_{\kac}f^{\otimes N}\pa{\bm{V}_N}\log\pa{f\pa{v_i}}d\sigma_N $$
$$\underset{\text{symmetry}}{=}N\int_{\kac}f^{\otimes N}\pa{\bm{V}_N}\log\pa{f\pa{v_1}}d\sigma_N\approx N\HH\pa{f|\mathcal{M}}$$
with $\MM(v)=\pa{2\pi}^{-1/2}e^{-v^2/2}$ appearing due to the integration of $\prod_{j\not=1}f(v_j)$ against $d\sigma_N$. The above tells us, at least informally, that if $\HH_N\pa{F_N(t)}$ converges exponentially to zero with a rate that is independent of $N$, then a simple rescaling by $N$ of the functional should give us an exponential convergence to equilibrium for the functional $\HH\pa{f(t)|\MM}$ in the limit equation. \\
The next natural question is, thus, \textit{how can we estimate quantitatively the convergence of $\HH_N$ to zero?}

The method we'll use to explore this question is known as \textit{the entropy method}. Let us motivate the ideas of this method by going back to spectral gaps.\\
Consider the equation
$$\partial_t u = Lu, $$
where $L$ is a semi-definite negative operator whose kernel is one dimensional. Let $u_{\infty}$ be a normalised basis for this kernel which corresponds to the unique equilibrium of the system.
The evolution of the natural distance, the $L^2$ norm, between the solution to the above,  $u(t)$, and the equilibrium $u_\infty$ is given by 
$$\frac{d}{dt}\norm{u(t)-u_\infty}^2 = 2\inner{u(t)-u_\infty, L\pa{u(t)-u_\infty}}.$$
The right hand side of the above expression brings to mind the fact that (in certain settings) we can find the eigenvalues of a self-adjoint operator $L$ by looking at $\inner{u,Lu}$ with $\norm{u}=1$. The spectral gap of $L$, $\lambda$, can be shown to equal
$$\lambda = \inf_{f\perp u_\infty,\;\norm{f}=1}\inner{f,-Lf}$$
and consequently we see that as long as $\inner{u(t),u_\infty}=1$ for $t\geq 0$, which is guaranteed if and only if $\inner{u(0),u_\infty}=1$\footnote{ $$\frac{d}{dt}\inner{u(t),u_\infty}=\inner{Lu(t),u_\infty}=\inner{u(t),Lu_\infty}=0.$$
	This is the conservation of mass property for the equation.}, we have that\footnote{we require this condition since only then do we have that $$u(t)-u_\infty = u(t)-\inner{u(t),u_\infty}u_\infty\in \br{u_\infty}^{\perp}.$$}
$$\frac{d}{dt}\norm{u(t)-u_\infty}^2 = 2\inner{u(t)-u_\infty, L\pa{u(t)-u_\infty}} \leq -2\lambda \norm{u(t)-u_\infty}^2.$$
This differential inequality results in an exponential convergence to equilibrium with explicit rate which is the spectral gap 
$$\norm{u(t)-u_\infty} \leq e^{-\lambda t}\norm{u(0)-u_\infty}.$$
At this point we notice that we managed to achieve a concrete rate of convergence by finding a \textit{functional inequality} (in our case based on a spectral property) that connected between the dissipation of our proposed distance and the distance itself. This is the key idea of the entropy method - we just need to perform it on the entropy.

Given an entropy for our system (or more precisely - a relative entropy), i.e. a Lyapnov functional $\mathcal{E}\pa{f|f_\infty}$\footnote{i.e. a functional such that $\frac{d}{dt}\mathcal{E}\pa{f(t)|f_\infty}\leq 0$ for any solution, $f(t)$, of the equation.} such that $\mathcal{E}\pa{f|f_\infty}=0$ if and only if $f=f_\infty$, the entropy method proceeds as follows:
\begin{itemize}
	\item Differentiating the entropy along the flow of the evolution equation we find that its dissipation can be written as $-\mathcal{D}\pa{f(t)}$, where the new \textit{non-negative} functional $\DD$ is known as the \textit{entropy production}.
	\item Putting the evolution of our system aside\footnote{in a sense we hope that our entropy and its production have truly captured the geometry of the problem and the relevant properties of the evolution.} we search for an explicit function $\Phi:\R_+\to \R_+$ such that 
	\begin{equation}\label{eq:entropy_method_functional_inequality}
		\DD\pa{f} \geq \Phi\pa{\EE\pa{f|f_\infty}}
	\end{equation}
	for all appropriate functions.
	\item Once the above functional inequality has been found, we recall that $\DD$ is connected to the dissipation of $\EE$ and conclude the following differential inequality
	$$\frac{d}{dt}\EE\pa{f(t)|f_\infty} \leq - \Phi\pa{\EE\pa{f(t)|f_\infty}}$$
	from which we conclude a \textit{quantitative expression} of the convergence to zero of $\EE\pa{f(t)|f_\infty}$.
\end{itemize}
The most desirable situation in the application of the entropy method, and in a sense the most physical one, is when $\Phi$ in $\eqref{eq:entropy_method_functional_inequality}$ is given by $\Phi(x)=\lambda x$ for some $\lambda>0$. We can think about this case as an ``entropic'' spectral gap as it implies the functional inequality 
$$\DD(f)\geq \lambda \EE\pa{f|f_\infty}$$
which results in an exponential convergence to equilibrium.

Applying the entropy method to our entropy $\HH_N$ under the flow of Kac's master equation \eqref{eq:master} we find that the entropy production in our case is given by
$$\DD\pa{F_N}=-\inner{\log F_N, \LN F_N}_{L^2\pa{\kac,d\sigma_N}}.$$
In order to achieve our desired exponential convergence of $\HH_N\pa{F_N}$\footnote{it is worth to note that $\HH_N$ is indeed a \textit{relative} functional. It measures the entropic distance of $F_N\dsn$ from $1\dsn$.} we consider the entropic spectral gap
\begin{equation}\nonumber
	\Gamma_N =\inf_{F_N\dsn\in \mathcal{P}\pa{\kac}}\frac{\mathcal{D}_N(F_N)}{\HH_N(F_N)}
\end{equation}
and, much like Kac's spectral gap, ask whether or not 
$$\Gamma=\inf_{N}\Gamma_N >0.$$
In his 2003 paper, \cite{V2003}, Villani has shown that\footnote{the formula presented here was achieved by some simplification of the argument of \cite{V2003} in the case of the Kac model we present here (see \cite{CCLLV2010} for more details).} 
$$\Gamma_N \geq \frac{2}{N-1}$$
and conjectured that it is of an order of magnitude $1/N$. If true, since by definition 
$$\HH_N\pa{F_N(t)} \leq e^{-\Gamma_N t}\HH_N\pa{F_N(0)},$$
Villani's conjecture would imply that in order to see significant decay in the entropy on Kac's sphere we must wait a time that is proportional to $N$ - exactly like our spectral gap instance. 

As we will see, the problem we're facing in this setting, unlike the spectral gap, is not the underlying distance - it is the production itself. 

An important step in the investigation of Villani's conjecture was made in the 2010 work of Carlen, Carvahlo, Le-Roux, Loss and Villlani, \cite{CCLLV2010} where the authors have managed to show that 
$$\Gamma=0.$$
Expanding on the ideas presented in that paper, Villani's conjecture has been shown to be essentially true in the work of the author \cite{Einav2011Kac} where he proved the following:
\begin{theorem}\label{thm:villani_conjecture}
	For any $0<\eta<1$ there exists an explicit constant $C_\eta$, that blows up as $\eta$ goes to $1$, such that
	\begin{equation}\nonumber
		\Gamma_N \leq \frac{C_\eta}{N^{\eta}}.
	\end{equation}
\end{theorem}

The intuition behind the proof of these results can be explained relatively simply: How can we create an ensemble that takes a very long to reach equilibrium? Consider an system in which most of the particles are fairly stable but there is a small amount of particles that is extremely energetic - carrying half the energy of the \textit{entire} system, for instance. If the ensemble is chaotic, i.e. ``almost independent'', it will take a very long time for these energetic particles to meet the slow moving particles and transfer some of their energy so that the entire system can equilibrate. 

Showing this precisely relies heavily on understanding chaoticity on Kac's sphere and on constructing a very natural type of chaotic state - the so-called \textit{conditioned tensorisation}.  

\section{Conditioned tensorisation - the notion of ``almost independence''}

\subsection{Conditioned tensorisations}
Looking back at the description of the mean field limit process, our definition of chaos, and its propagation, one question which we have yet to address becomes apparent: Are there chaotic states at all?

Had our phase state been of the form $\mathcal{X}^N$, for some Polish space $\mathcal{X}$ (for instance $\R^N$) then chaoticity, which seem to mean ``almost independence'',\footnote{we do need to be quite careful here. The independence that chaos guarantees is always ``capped'' at a \textit{given} finite marginal. We have no idea what happens with the correlations between a number of elements that is of order of $N$.} could be attained by considering an independent state, i.e. a tensorisation of a given probability density $f$ on $\mathcal{X}$, $F_N=f^{\otimes N}$. This, however, is impossible on $\kac$ since the velocity variables can not be independent. Can we, though, find a chaotic state on Kac's sphere that looks as if it is almost a tensorised state?

One natural approach presents itself readily: Given a probability density function on $\R$, $f$, we can tensorise it to get an independent state on $\R^N$ and then \textit{restrict it to Kac's sphere} in hope that not too many correlations will manifest. In other words, given a probability density function $f$ on $\R$ we define
\begin{equation}\label{eq:conditioned_tensorisation}
	F_N\pa{\bm{V}_N} = \frac{f^{\otimes N}\pa{\bm{V}_N}}{\mathcal{Z}_N\pa{f,\sqrt{N}}},\quad \text{where}\quad\mathcal{Z}_N\pa{f,\sqrt{N}}= \int_{\kac}f^{\otimes N}\pa{\bm{V}_N}\dsn.
\end{equation}
$F_N$ of the above form is called the \textit{conditioned tensorisation of $f$} and the function $\ZZ_N$ is known as \textit{the normalisation function}.

A few immediate questions come to mind when looking at the definition of these states:
\begin{itemize}
	\item Are they well defined? Clearly if $f$ is in addition continuous then $\ZN$ is well defined and $F_N$ is indeed a probability density on $\kac$ with respect to $\dsn$. What happens, however, if $f$ is only measurable? $\ZN$ is an integration over a set of measure zero with respect to $dx$ and as such we require some sort of trace theorem. 
	\item Are there any restriction on $f$? From our set up of Kac's model, and in particular our choice of $\kac$ as the underlying space, we know that the random variable that is associated to $f$, say $\mathscr{V}$, is supposed to have a unit kinetic energy. This means that to be consistent we must have that
	$$\int_{\R}v^2 f(v)dv =1.$$
	Since the above implies that the independent random ensemble $\pa{\mathscr{V}_1,\dots,\mathscr{V}_N}$ (whose values are in $\R^N$) has a kinetic energy of value $N$ we have a chance, according to the law of large numbers, that the ensemble is concentrated enough on $\kac$ for its restriction to Kac's sphere to make sense. 
	\item Is $F_N$ indeed chaotic? It is possible to show (see \cite{CCLLV2010,Einav2011Kac} and \cite{CE2013Levi} for more information) that 
	\begin{equation}\label{eq:marginal_conditioned_tensoriation}
		\begin{split}
			F_{N,k}\pa{v_1,\dots,v_k} =& \frac{\abs{\sph^{N-k-1}}}{\abs{\sph^{N-1}}}\;\frac{1}{N^{\frac{N-2}{2}}}\pa{N-\sum_{i=1}^k v_i^2}_+^{\frac{N-k-2}{2}}\\
			& \frac{\ZZ_{N-k}\pa{f,\sqrt{\pa{N-\sum_{i=1}^k v_i^2}_+}}}{\ZN}f^{\otimes k}\pa{v_1,\dots,v_k}
		\end{split}
	\end{equation}
	with 
	\begin{equation}\label{eq:concentration_function}
		\ZZ_n\pa{f,r} = \int_{\sph ^{m-1}\pa{r}}f^{\otimes n}\pa{\bm{V}_n} d\sigma_{n,r}
	\end{equation}
	and where $d\sigma_{n,r}$ is the uniform probability measure on $\sph^{n-1}\pa{r}$. In particular $\dsn=d\sigma_{N,\sqrt{N}}$.\\
\end{itemize}

Conditioned tensorisation are not a new concept. In fact, Kac mentioned them in his work \cite{Kac1956}. However, the conditions that he needed to impose on $f$ in his discussion were very stringent. A better understanding of these states, on how one can consider the trace operation which we seem to employ by defining them, and on the question of chaoticity was achieved in the work of Carlen, Carvalho, Le Roux, Loss and Villani, \cite{CCLLV2010}, which included some beautiful probabilistic observations and ideas. 

To truly understand conditioned tensorisation we must understand the behaviour of the normalisation function $\ZZ_N\pa{f,r}$ for all $r\in \rpa{0,\sqrt{N}}$.

By its definition, we expect that $\ZZ_n\pa{f,r}$ will tell us, in some sense, how concentrated $f^{\otimes n}$ is on $\sph^{n-1}\pa{r}$ - up to some geometric factors. A more careful look, however, reveals that this concentration is intimately connected to the probability density of $\sum_{i=1}^n \mathscr{V}_i^2$ - the random variable of the kinetic energy of the ensemble.

Indeed, if $h$ is the probability density of the variable $\mathscr{V}^2$, attained from the probability density of $\mathscr{V}$, $f$, by
\begin{equation}\label{eq:h_and_f}
	h(r) = \frac{f\pa{\sqrt{r}}+f\pa{-\sqrt{r}}}{2r},\qquad \forall r>0
\end{equation}
we have that for any radial function on $\R^n$, $\phi(\bm{X}_n)=\varphi\pa{\abs{\bm{X}_n}}$
$$\int_{\R^n}\phi\pa{\bm{V}_n}f^{\otimes n }\pa{\bm{V}_n}d\bm{V}_n = \int_{r=0}^\infty \varphi(r) \abs{\sph^{n-1}}r^{n-1}\underbrace{\int_{\sph^{n-1}\pa{r}}f^{\otimes n }\pa{\bm{V}_n}d\sigma_{n,r}}_{=\ZZ_n\pa{f,r}}.$$
On the other hand
$$\int_{\R^n}\phi\pa{\bm{V}_n}f^{\otimes n }\pa{\bm{V}_n}d\bm{V}_n  = \int_{0}^{\infty} \varphi\pa{\sqrt{t}}h^{\otimes n}\pa{t}dr=\int_{0}^\infty \varphi(r) 2rh^{\otimes n}\pa{r^2}dr, $$
where we have used the independence of $\br{\mathscr{V}_i^2}_{i=1,\dots,n}$ which follows from that of $\br{\mathscr{V}_i}_{i=1,\dots,n}$.
From all the above we can conclude that
\begin{equation}\label{eq:Z_n_and_h}
	\ZZ_n\pa{f,\sqrt{r}} = \frac{2h^{\otimes n}\pa{r}}{r^{\frac{n-2}{2}}\abs{\sph^{n-1}}}.
\end{equation}
Formula \eqref{eq:Z_n_and_h} gives us a way to go forward in investigating $\ZZ_N$ - we would like to find a \textit{local Central Limit Theorem}\footnote{i.e. a central limit theorem for the probability density of the independent sum.} for the tensorisation of $h$. One version of such theorem, expressed via $f$, was shown in \cite{CCLLV2010}:
\begin{theorem}\label{thm:Z_n_approximation_original}
	Let $f$ be a probability density function on $\mathbb{R}$ such that $f\in L^p\pa{\mathbb{R}}$ for some $p>1$. Assume in addition that 
	$$\int_{\mathbb{R}} v^2 f(v)dv=1 \quad \text{and}\quad  \int_{\mathbb{R}}v^4 f(v)dv <\infty.$$ Then $\mathcal{Z}_N(f,\sqrt{r})$ is well defined for any $r\in \rpa{0,N}$, and
	\begin{equation}\label{eq:Z_n_approximation_original}
		\mathcal{Z}_N(f,\sqrt{r})=\frac{2}{\sqrt{N}\Sigma \left\lvert \mathbb{S}^{N-1}\right\rvert  r^{\frac{N-2}{2}}}\left( \frac{e^{-\frac{(r-N)^2}{2N\Sigma^2}}}{\sqrt{2\pi}}+\lambda_N(r) \right),
	\end{equation}
	with
	\begin{equation}\nonumber
		\sup_{r\in [0,N]} \abs{\lambda_N(r)}\underset{N\rightarrow\infty}{\longrightarrow}0,
	\end{equation}
	and where $\Sigma^2 = \int_{\mathbb{R}}v^4f(v)dv - 1$.
\end{theorem}

Besides the fact that the above theorem gives us relatively natural conditions under which conditioned tensorisation of a probability density $f$ makes sense, plugging \eqref{eq:Z_n_approximation_original} in \eqref{eq:marginal_conditioned_tensoriation} shows almost immediately the such states are $f-$chaotic.

The usefulness of this local CLT doesn't end there - it shows that conditioned tensorisation are indeed very natural in Kac's setting as their entropy and production scale linearly in $N$. This, in turn, will be the crucial step we need to show Carlen et. al's result and its extension to Theorem \ref{thm:villani_conjecture}.
\subsection{Conditioned tensorisation and entropic convergence to equilibrium}
Consider the conditioned tensorisarion of $f$, $F_N$, where $f$ satisfies the conditions of Theorem \ref{thm:Z_n_approximation_original}. Then, as can be seen in \cite{CCLLV2010,Einav2011Kac}, 
\begin{equation}\label{eq:entropy_of_conditioned}
	\begin{aligned}
		\HH_N\pa{F_N} = &\frac{N\abs{\sph^{N-2}}}{N^{\frac{N-2}{2}}\abs{\sph^{N-1}}\ZN}\Bigg[\int_{-\sqrt{N}}^{N}f(v)\log f(v) \pa{N-v^2}^{\frac{N-3}{2}}\\
		&\ZZ_{N-1}\pa{f,\sqrt{N-v^2}}dv\Bigg] - \log \pa{\ZN}.
	\end{aligned}
\end{equation}
and
\begin{equation}\label{eq:production_of_conditioned}
	\begin{aligned}
		\DD_N\pa{F_N} = &\frac{N\abs{\sph^{N-3}}}{4\pi N^{\frac{N-2}{2}}\abs{\sph^{N-1}}\ZN}\\
		&\int_{0}^{2\pi}d\theta \int_{v_1^2+v_2^2 \leq N}\psi\pa{f(v_1)f(v_2),f\pa{v_1\pa{\theta}}f\pa{v_2\pa{\theta}}}\\
		&\pa{N-v_1^2-v_2^2}^{\frac{N-4}{2}}\ZZ_{N-2}\pa{f,\sqrt{N-v_1^2-v_2^2}}dv_1dv_2.
	\end{aligned}
\end{equation}
where 
\begin{equation}\label{eq:production_generator}
	\psi\pa{x,y} = \pa{x-y}\log \pa{\frac{x}{y}}.
\end{equation}

Utilising Theorem \ref{thm:Z_n_approximation_original} we conclude that

\begin{theorem}\label{thm:entropic_convergence}
	Let $f$ be a probability density on $\R$. Assume that the conditions of Theorem \ref{thm:Z_n_approximation_original} hold and let $F_N$ be the conditioned tensorisation of $f$.  Then
	\begin{equation}\nonumber
		\lim_{N\to\infty}\frac{\HH_N\pa{F_N}}{N}=\int_{\R}f(v)\log f(v) dv +\frac{1}{2}+\frac{\log 2\pi}{2}=\HH\pa{f|\MM}
	\end{equation}
	and 
	\begin{equation}\nonumber
		\lim_{N\to\infty}\frac{\DD_N\pa{F_N}}{N}=\frac{1}{4\pi}\int_{\R^2 \times [0,2\pi]}\psi\pa{f(v_1)f(v_2),f\pa{v_1\pa{\theta}}f\pa{v_2\pa{\theta}}}dv_1dv_2d\theta=\frac{\DD\pa{f}}{2}
	\end{equation}
	where
	\begin{equation}\label{eq:def_of_DD}
		\begin{aligned}
			&\DD(f)=\frac{1}{2\pi}\int_{\R\times [0,2\pi]}\psi\pa{f(v)f(v_\ast),f\pa{v\pa{\theta}}f\pa{v_{\ast}\pa{\theta}}}dv dv_\ast d\theta,
		\end{aligned}		
	\end{equation}
	with $v_1\pa{\theta}$ and $v_2\pa{\theta}$, as well as $v\pa{\theta}$ and $v_\ast\pa{\theta}$, given by the same formula as \eqref{eq:collision_velocitis}.
\end{theorem}

Besides further cementing the status of conditioned tensorisation as a worthy and natural family to explore in the setting of Kac's model, Theorem \ref{thm:entropic_convergence} is the key to showing Carlen et. al. claim that 
$$\inf_{N\in\N}\Gamma_N = \inf_{N\in\N}\inf_{F_N\dsn\in \mathcal{P}\pa{\kac}}\frac{\DD_N\pa{F_N}}{\HH_N\pa{F_N}}=0.$$
Indeed, as long as $f$ satisfies the condition of Theorem \ref{thm:Z_n_approximation_original}, we find that by considering the family of condition tensorisation of $f$, which we will denote by $\rpa{f^{\otimes N}}$, that
\begin{equation}\label{eq:entropic_spectral_uniform_estimate}
	\inf_{N\in\N}\Gamma_N \leq \lim_{N\to\infty} \frac{\DD_N\pa{\rpa{f^{\otimes N}}}}{\HH_N\pa{\rpa{f^{\otimes N}}}} = \frac{\DD\pa{f}}{2\HH\pa{f|\MM}}.
\end{equation}
Following up on the intuition we presented in the end of \S\ref{subsec:entropy} we consider a ``generator'' $f$ that somehow gives an average particle very low probability to be very energetic and almost full probability to be stable. A natural choice is given by\footnote{note that $\int_{\R} v^2 \MM_{a}(v)dv = a$ which shows that each part of $f_{\delta}$ carries the same kinetic energy.}
\begin{equation}\label{eq:f_carlen_choice}
	f_{\delta}(v) = \delta \MM_{\frac{1}{2\delta}}(v) + \pa{1-\delta}\MM_{\frac{1}{2\pa{1-\delta}}}(v)
\end{equation}
where $\MM_a(v) = \pa{2\pi a}^{-1/2}e^{-\frac{v^2}{2a}}$ and $0<\delta<1$. Plugging it in \eqref{eq:entropic_spectral_uniform_estimate} yields the estimate
$$\inf_{N\in\N}\Gamma_N  \leq -C\delta \log \pa{\delta},$$
where $C$ is a fixed bounded constant. Taking $\delta$ to zero gives the desired result.

The above idea, however, is not enough to give an explicit bound on every $\Gamma_N$ as is expressed in Theorem \ref{thm:villani_conjecture}. The reason behind this lies in the choice \eqref{eq:f_carlen_choice}. While it captures the intuition that an average particle has small probability to be very energetic, this probability - and consequently the (average) portion of the particles in the ensemble that will be very energetic, is \textit{fixed} and involves a non-negligible amount of particles. This is why we are getting a uniform in $N$ bound on $\Gamma_N$. In order to utilise the same idea and get an $N-$dependent bound for $\Gamma_N$ at the same time, we need to allow the number of particles that carry the high energy to be significantly smaller than $N$. In other words, we need to let the probability that an average particle will be energetic to decrease with $N$. We will, thus, consider a generator of the form 
\begin{equation}\label{eq:f_my_choice}
	f_{\delta_N}(v) = \delta_N \MM_{\frac{1}{2\delta_N}}(v) + \pa{1-\delta_N}\MM_{\frac{1}{2\pa{1-\delta_N}}}(v)
\end{equation}
for some choice of $\br{\delta_N}_{N\in\N}$ that goes to zero as $N$ goes to infinity. \\
While this seems not too different to what we've done before, there is a crucial ingredient in being able to use the above idea that is new. The result of Carlen et. al.'s relied heavily on the local-CLT theorem which is expressed in Theorem \ref{thm:Z_n_approximation_original}. It is unclear if the same result holds in this case due to the following two reasons:  
\begin{enumerate}[1)]
	\item Theorem \ref{thm:Z_n_approximation_original} has considered the expression \eqref{eq:Z_n_and_h}. Since our function $f_N$ depends on $N$ we need to be careful when we tensorise it (or the associated $h_N$) a number of times that is of order $N$. Such orders are particularly important to us as the rescaled entropy and productions depend on $\ZZ_N$, $\ZZ_{N-1}$ and $\ZZ_{N-2}$.
	\item A straight forward calculation shows that 
	$$\int_{\R} v^4 f_{\delta_N}(v)dv =\frac{3}{4\delta_N\pa{1-\delta_N}} $$
	which is not bounded as $N$ goes to infinity. This implies that the conditions of Theorem \ref{thm:Z_n_approximation_original} do not hold uniformly in $N$.
\end{enumerate}

The resolution to the above issues was given in \cite{Einav2011Kac} where the author developed, under certain conditions, a new local-CLT that allowed for dependence in $N$ in the probability density for the generator of the conditioned tensorisation. The theorem allowed for some blow up in $N$ for the fourth moment We will not state it here, as it is quite technical, and refer the interested reader to the aforementioned paper. 

With this local-CLT the author has managed to show the following concentration estimate for the normalisation function of $f_{\delta_N}$, defined by \eqref{eq:f_my_choice}:

\begin{theorem} \label{thm:approximation_of_ZZ_N_N_dependent}
	Let $\br{\delta_N}_{N\in\N}$ be a sequence of positive numbers that satisfies
	\begin{itemize}
		\item $\delta_N$ is dominated by a some power of $N$.
		\item There exists $\beta>0$ such that 
		$$\delta_{N}^{1+2\beta} N\underset{N\rightarrow\infty}{\longrightarrow}\infty$$
		and 
		$$\delta_{N}^{1+3\beta} N\underset{N\rightarrow\infty}{\longrightarrow}0.$$
	\end{itemize}
	Then, for any fixed $j$ 
	\begin{equation}
		\begin{aligned}
			Z_{N-j}\left(f_{\delta_{N}},\sqrt{u}\right)=&\frac{2}{\sqrt{N-j}\cdot\Sigma_{\delta_{N}}\cdot|\mathbb{S}^{N-j-1}|u^{\frac{N-j}{2}-1}}\\
			&\left(\frac{e^{-\frac{\left(u-N+j\right)^{2}}{2(N-j)\Sigma_{\delta_{N}}^{2}}}}{\sqrt{2\pi}}+\lambda_{j}(N-j,u)\right)
		\end{aligned}
	\end{equation}
	with
	$$\lim_{N\to\infty}\sup_{u\in\mathbb{R}}\left|\lambda_{j}(N-j,u)\right|=0,$$
	where 
	$$\Sigma_{\delta_N}^2 = \int_{\R}v^4 \delta_N(v)dv -1=\frac{3}{4\delta_N\pa{1-\delta_N}}-1.$$ 
\end{theorem}

Using the above theorem for the conditioned tensorisation of $f_{\delta_N}$ with $\delta_N = N^{2\beta -1}$ and $0<\beta <1/6$ chosen arbitrarily, together with \eqref{eq:entropy_of_conditioned} and \eqref{eq:production_of_conditioned} gives us the proof of Theorem \ref{thm:villani_conjecture}

\begin{remark}\label{rem:additional_CLT}
	We would like to mention at this point that additional extensions of the local-CLT presented in Theorem \ref{thm:Z_n_approximation_original} have been explored. In particular, in their work \cite{CE2013Levi} the authors have developed a similar concentration result for the normalisation function of a probability density function $f$ such that $f\in L^p\pa{\R}$ for some $p>1$, $\int_{\R}v^2f(v)dv=1$ and 
	$$\int_{-\sqrt{x}}^{\sqrt{x}}v^4 f(v) \underset{x\to\infty}{\sim}C x^{2-\alpha}$$
	for some $C>0$ and $1<\alpha<2$. In this case the concentration of $\ZZ_N$ is measured not with a Gaussian, as in theorems \ref{thm:Z_n_approximation_original} and \ref{thm:approximation_of_ZZ_N_N_dependent}, but with a the attractor that is associated to an $\alpha-$stable L\'evi process. 
\end{remark}

While our attempt to utilise the mean field limit approach with our entropic study of convergence to equilibrium in Kac's model has not been successful, we have found intuitive families of chaotic states - conditioned tensorisations. These families are a crucial ingredient in what comes next - the (somewhat) fulfilment of Kac's vision.

\section{Kac's vision fulfilled}\label{sec:vision_fulfilled}
\subsection{New notions of chaoticity}\label{subsec:new_chaos}
When we discussed the entropic study of the convergence to equilibrium in Kac's model back in \S\ref{subsec:entropy}, we motivated our decision to choose the entropy as our ``distance'' with the intuitive reasoning that the entropy of chaotic states should scale linearly in $N$. At no point, however, did we prove that it is true in general. We did show in our previous section that conditioned tensorisations do satisfy this property. Moreover - their entropy production \textit{also} scales linearly in $N$. These observations motivate the following definitions
\begin{definition}\label{def:chaos_extended}
	Let  $\br{F_N}_{N\in\N}$ be an $f-$chaotic family of probability densities in $\mathcal{P}\pa{\kac,\dsn}$. 
	\begin{enumerate}[(i)]
		\item We say that $\br{F_N}_{N\in\N}$ is $f-$\textit{entropically chaotic} if 
		\begin{equation}\label{eq:entropic_chaotic}
			\lim_{N\to\infty}\frac{\HH_N\pa{F_N}}{N} = \HH\pa{f|\MM}.
		\end{equation}
		\item We say that $\br{F_N}_{N\in\N}$ is $f-$\textit{strongly entropic chaotic} if in addition to the above
		\begin{equation}\label{eq:strongly_entropic_chaotic}
			\lim_{N\to\infty}\frac{\DD_N\pa{F_N}}{N} = \frac{\DD\pa{f}}{2},
		\end{equation}
		where $\DD\pa{f}$ is given in \eqref{eq:def_of_DD}.
	\end{enumerate}
\end{definition} 

The notion of entropic chaoticity was originally defined and investigated in \cite{CCLLV2010}. Further work, such that of Hauray and Mischler \cite{HM2014}, have expanded on this study and even showed that under certain conditions the chaoticity of a family $\br{F_N}_{N\in\N}$ can follow from condition \eqref{eq:entropic_chaotic}. 

The reason why we identified and defined these new stronger notion of chaoticity is that on closer examination of our approach to the investigation of the convergence to equilibrium, such notions offer a different avenue for attempting to deduce information for the Boltzmann-Kac equation \eqref{eq:boltzmann_kac} from Kac's model. 

In \S\ref{sec:kac_model} we tried to use the entropy method to find an explicit convergence rate for $\HH_N\pa{F_N(t)}$ in hope that we can then pass it to $\HH\pa{f(t)|\MM}$ by rescaling Kac's entropy and taking $N$ to infinity - implicitly assuming entropic chaoticity. By doing so, however, we have lost some of the geometry of the process as we've ``thrown away'' the functional inequality we've found and only kept its conseqeunce.

The idea that ``nice enough'' chaotic states, like strongly entropic ones, scale linearly in $N$ both in entropy and its production give rise to the intuition that we should look for a functional inequality of the form

$$\frac{\DD_{N}\pa{F_N}}{N} \geq \Phi\pa{\frac{\HH_N\pa{F_N}}{N}},$$
for some function $\Phi$.


\subsection{What do we expect - the entropy method for the Boltzmann equation}\label{subsec:what_to_expect_from_boltzmann}
Since the time when Kac presented his model and expressed his hope to infer information on the Boltzmann equation from it, many techniques to deal with non-linear equations have been developed. In fact, nowadays many questions that pertain to various aspects of Kac's model are more complicated to solve than similar questions for the Boltzmann equation. 

Before we delve into exploring functional inequalities of rescaled entropy and its production, we want to pause and reflect on the study of the entropy method in the Boltzmann equation. 

The paper where Villani's managed to show that the entropic spectral gap $\Gamma_N$ is bounded from below by $2/(N-1)$, \cite{V2003}, was in fact mostly dedicated to the study of the entropy method for the Boltzmann equation. \amit{In their work \cite{BC1999}, Bobylev and Cercignani have} managed to show that when the collision kernel $B$ is given by 
$$B\pa{v,v_\ast,\sigma}=1+\abs{v-v_\ast}^\gamma$$
\amit{with $\gamma\in [0,2)$} then one finds that for the spatially homoegenous Boltzmann equation
$$\inf_{fdv \in \mathcal{P}\pa{\R^d}}\frac{\mathfrak{D}_{\gamma}\pa{f}}{\HH\pa{f|\MM}}=0,$$
where 
$$\mathfrak{D}_\gamma \pa{f} = \frac{1}{4}\int_{\R^d\times\R^d\times \sph^{d-1}}\pa{1+\abs{v-v_\ast}^\gamma}\psi\pa{f\pa{v}f\pa{v_\ast},f\pa{v^\prime, v^\prime_{\ast}}}dvdv_\ast d\sigma, $$
with $\psi$ defined as in \eqref{eq:production_generator}. Note that this result is consistent with the study on Kac's sphere.

\amit{Inspired by the works of Carlen and Carvahlo \cite{CC1992,CC1994},} \amit{Toscani and} Villani \amit{managed to show in \cite{TV1999}} that if one assumes that $f$ is smooth enough and has a Maxwellian-type lower bound (we leave the precise conditions to \cite{TV1999}) we have that for any $\epsilon>0$ there exists a constant $C_{\epsilon}(f)$ that depends on $f$ such that 
\begin{equation}\label{eq:villani_almost_spectral}
	\mathfrak{D}_{\gamma}(f) \geq C_{\epsilon}(f) \HH\pa{f|\MM}^{1+\epsilon}
\end{equation}
giving an ``almost'' entropic gap result. \amit{In his work \cite{V2003}, Villani has extended the result to various other collusion kernels as well as have shown the existence of an entropic spectral gap for the super quadratic collision kernel (i.e. when $\gamma=2$).}

Following on the above, and on the intimate interplay between Kac's model and the Boltzmann equation, it is reasonable for us to seek inequalities of the form \eqref{eq:villani_almost_spectral} on the level of Kac's sphere.

\subsection{General Kac's models and the ``correct'' inequalities}\label{subsec:rescaled_inequalities}

In the context of Boltzmann equation, Kac's original model fitted into the framework of Maxwellian molecules (i.e., the case where the collision kernel $B$ is constant). Extensions of this model can be made, at least formally, to the case where the collisions in Kac's process occurs with an intensity that depends on the scattering (i.e. on $\theta$) and the velocities of the colliding particles (i.e. on $v_i$ and $v_j$ if the particles that collided were the $i-$th and $j-$th ones)\footnote{in terms of a jump processes, the dependence on $v_i$ and $v_j$ is slightly problematic but we continue as if no additional justification for the proposed model is needed.}. 

The natural extension to Kac's master equation which we will consider here is given by
\begin{equation}\label{eq:master_kac_extended}
	\partial_t F_N \pa{t,\bm{V}_N}=\mathcal{L}_{N,\gamma} F_N,  \qquad \bm{V}_N\in \mathbb{S}^{N-1}\pa{\sqrt{N}}
\end{equation}
where
$$\mathcal{L}_{N,\gamma}F\pa{\bm{V}_N}=-\frac{N}{2\pi}\frac{1}{\pa{\begin{tabular}{c} $N$ \\ $2$ \end{tabular}}}\sum_{i<j}\pa{1+\pa{v_i^2+v_j^2}}^{\gamma}
\int_0^{2\pi}\pa{F_N\pa{\bm{V}_N}-F_N\pa{R_{i,j,\theta}\pa{\bm{V}_N}}}d\theta,$$
and $0\leq \gamma \leq 1$.

Under the assumption of the propagation of chaos we can repeat the mean field approach in this model to find the limit equation
\begin{equation}\label{eq:kac_boltzmann_extended}
	\begin{aligned}
		\partial_t f(v) =  \frac{1}{\pi}\int_{-\pi}^{\pi} &\pa{1+ (v^2+w^2)}^{\gamma} \pa{f(v(\theta))f(w(\theta))
			- f(v)f(w)} d\theta dw,
	\end{aligned}
\end{equation}
where $v\pa{\theta}$ and $w\pa{\theta}$ are given by formula \eqref{eq:collision_velocitis}.

Much like in Kac's original model we can try and explore entropic convergence to equilibrium of \eqref{eq:master_kac_extended} by finding a connection between the entropy of the system, $\HH_N\pa{F_N}$, and its production
$$\DD_{N,\gamma} = -\inner{\log F_N, \mathcal{L}_{N,\gamma} F_N}_{L^2\pa{\kac,d\sigma_N}}.$$
A straight forward calculation finds that
\begin{equation}\label{eq:def_of_DD_N_gamma}
	\begin{split}
		\DD_{N,\gamma}(F_N)=\frac{1}{4\pi}\frac{N}{\pa{\begin{tabular}{c} $N$ \\ $2$ \end{tabular}}}&\sum_{i<j}\int_{\kac}\int_0^{2\pi}\pa{1+\pa{v_i^2+v_j^2}}^\gamma\\
		&\psi\pa{F_N\pa{\bm{V}_N},F_N\pa{ R_{i,j,\theta}\pa{\bm{V}_N}}}d\sigma_Nd\theta.
	\end{split}
\end{equation}
where $\psi\pa{x,y}=\pa{x-y}\log \pa{x/y}$ (as defined in \eqref{eq:production_generator}). 

Following on the discussion in \S\ref{subsec:what_to_expect_from_boltzmann}  we are encouraged to search for a functional inequalities of the form
%
\begin{equation}\label{eq:desired_rescaled_inequality}
	\frac{\DD_{N,\gamma}\pa{F_N}}{N} \geq C_{\epsilon}\pa{\frac{\HH_N\pa{F_N}}{N}}^{1+\epsilon}
\end{equation}
for a given, hopefully arbitrary, $\epsilon>0$. 

We naturally don't expect this to hold for all chaotic states. However, something can still be said in the general case: In his work \cite{V2003}, Villani has shown the following inequality:
\begin{equation}\label{eq:villani_genera_production_ratio}
	\DD_{N,\gamma}(F_N) \geq C_\gamma \frac{\HH_N(F_N)}{N^{1-\gamma}}.
\end{equation}
This inequality, however, is not of the form \eqref{eq:desired_rescaled_inequality} unless $\gamma=1$.

To try and find conditions under which we expect to have an inequality of the form \eqref{eq:desired_rescaled_inequality}, let us look again at the intuitive idea that chaotic states $F_N$ behaves like $f^{\otimes N}$. Since we've already given an intuitive outlook in \S\ref{subsec:entropy} to why $\HH_N\pa{F_N}\approx N\HH\pa{f|\MM}$ we turn our attention to $\DD_{N,\gamma}$. 

Following on \eqref{eq:def_of_DD_N_gamma} and using symmetry we see that
$$\frac{\DD_{N,\gamma}\pa{F_N}}{N}=\frac{1}{4\pi}\int_{\kac}\int_0^{2\pi}\pa{1+\pa{v_1^2+v_2^2}}^\gamma
\psi\pa{F_N\pa{\bm{V}_N},F_N\pa{ R_{1,2,\theta}\pa{\bm{V}_N}}}d\sigma_Nd\theta.$$
Plugging in $F_N\approx f^{\otimes N}$ we get
\begin{equation}\nonumber
	\begin{aligned}
		\psi&\pa{F_N\pa{\bm{V}_N},F_N\pa{ R_{1,2,\theta}\pa{\bm{V}_N}}} \\
		&\approx \pa{F_N\pa{\bm{V}_N}-F_N\pa{ R_{1,2,\theta}\pa{\bm{V}_N}}}\log\pa{\frac{f\pa{v_1\pa{\theta}}f\pa{v_2\pa{\theta}}}{f\pa{v_1}f\pa{v_2}}}
	\end{aligned}
\end{equation}
and notice that the dependence in the dimension is completely cancelled in the logarithm, leaving only an $L^1\pa{\kac,\dsn}$ like term. This, in a sense, is exactly why the scaling works for such states. As $F_N$ can't be expected to look like $f^{\otimes N}$, besides for conditioned tensorisation, we might need to impose stringer control on the integral that appears above. This motivates the following definition:

\begin{definition}\label{def:log_power}
	We say that a family of symmetric probability densities $\br{F_N}_{N\in\mathbb{N}}$ on $\pa{\kac,\dsn}$ has the log-power property of order $\beta>0$, or is  log-power of order $\beta$ in short, if there exists $\mathcal{C}_{\beta}>0$ such that
	\begin{equation}\label{eq:log_power}
		\sup_{N\in\N}\frac{1}{2\pi}\int_{0}^{2\pi}\int_{\kac} \psi_\beta\pa{F_N\pa{\bm{V}_N},F_N\pa{ R_{1,2,\theta}\pa{\bm{V}_N}}}d\sigma^N d\theta \leq \mathcal{C}_{\beta}^{1+\beta}
	\end{equation}
	where $\psi_\beta(x,y)=\abs{x-y}\abs{\log\pa{\frac{x}{y}}}^{1+\beta}$. We will call the constant $\mathcal{C}_{\beta}$ the \textit{log power} constant.
	\end{definition}
	
	The log power property is indeed the right ingredient we needed to achieve an inequality of the form \eqref{eq:desired_rescaled_inequality}. On its own it is not enough, but by adding a natural requirement of a uniform in $N$ moment control for $F_{N,1}$, Carlen, Carvalho and the author of this paper have managed to show the following  in \cite{CCE2018}:
	\begin{theorem}\label{thm:rescaled_inequality}
\sloppy Let $\br{F_N}_{N\in\N}$ be a family of symmetric density functions on $\pa{\kac,\dsn}$ . Assume in addition that there exists some $\beta>0$ such that $\br{F_N}_{N\in\N}$ is log-power of order $\beta$ and that there exists $k>1+1/\beta$ such that 
\begin{equation}\label{eq:moment_control}
	M_{2k} = \sup_{N\in\N}M_{2k}\pa{F_{N,1}}<\infty,
\end{equation}
where
\begin{equation}\label{eq:moment}
	M_{2k}\pa{F_{N,1}} = \int_{\R}v^{2k}F_{N,1}(v)dv = \int_{\kac}v_1^{2k}F_N\pa{\bm{V}_N}\dsn.
\end{equation}
Then, there exists an explicit $C_{k,\gamma,\beta}>0$ that depends only on $\beta$, the log power constant of $\br{F_N}_{N\in\N}$, $\mathcal{C}_{\beta}$, and $M_{2k}$ such that
\begin{equation}\label{eq:desired_inequality_thm}
	\begin{gathered}
		\frac{\DD_{N,\gamma}(F_N)}{N} \geq C_{k,\gamma,\beta}\pa{\frac{\HH_N(F_N)}{N}}^{1+\frac{(1-\gamma)(1+\beta)}{k\beta-(1+\beta)}}.
	\end{gathered}
\end{equation}
\end{theorem}

A few observation on the above theorem:
\begin{itemize}
\item The condition $k>1+1/\beta$ gives us an interplay between our moment control, $k$, and the log-power parameter, $\beta$. The tighter control we have on the integral \eqref{eq:log_power}, the less moment we need. 
\item Moments conditions are very natural in the setting of the Boltzmann equation and many other kinetic equations so we shouldn't be too surprised by the appearance of $M_{2k}\pa{F_{N,1}}$. 
\item For a given $\epsilon>0$ we see that if $F_N$ is log-power of order $\beta$ and we have moment control of high enough order then we can recover the inequality
$$\frac{\DD_{N,\gamma}(F_N)}{N} \geq C_{\epsilon}\pa{\frac{\HH_N(F_N)}{N}}^{1+\epsilon}.$$
\end{itemize}

The difficulty in showing Theorem \ref{thm:rescaled_inequality} was mostly in finding the right condition under which it holds, i.e. the notion of log-power of some order. Once we have convinced ourselves that it is the right ingredient for the desired functional inequality, the proof is a fairly straight forward combination of \eqref{eq:villani_genera_production_ratio} for $\gamma=1$ and interpolation techniques. We would like to mention at this point that using a similar technique for the Boltzmann equation is exactly how Villani has shown \eqref{eq:villani_almost_spectral} in \cite{V2003}. However, this, at least in the author's opinion, is much more involved than the proof of Theorem \ref{thm:rescaled_inequality}.

Due to its relative simplicity, let us give the proof to this theorem here:
\begin{proof}[Proof of Theorem \ref{thm:rescaled_inequality}]
For any $\lambda >0$ we define the set
$$A_{\lambda}=\br{v_1,v_2\in \R \;\Big|\; 1+v_1^2+v_2^2 > \lambda}.$$ 
Using the symmetry of $F_N$ we find that
\begin{equation}\nonumber
	\begin{split}
		\DD_{N,1}(F_N)\leq &\lambda^{1-\gamma} D_{N,\gamma}(F_N) \\
		&+\frac{N}{4\pi}\int_{\kac \cap A_{\lambda}}\int_{0}^{2\pi}\pa{1+v_1^2+v_2^2}  
		\psi\pa{F_N\pa{\bm{V}_N},F_N\pa{ R_{1,2,\theta}\pa{\bm{V}_N}}}d\sigma_N d\theta\\
	\end{split}
\end{equation}

\begin{equation}\nonumber
	\begin{split}
		\leq &\lambda^{1-\gamma} D_{N,\gamma}(F_N) \\
		+ 	 \frac{N}{2}&\pa{\frac{1}{2\pi}\int_{0}^{2\pi}\int_{\kac}\abs{\log\pa{\frac{F_N\pa{ R_{1,2,\theta}\pa{\bm{V}_N}}}{F_N\pa{\bm{V}_N}}}}^{1+\beta}\abs{F_N\pa{\bm{V}_N}-F_N\pa{ R_{1,2,\theta}\pa{\bm{V}_N}}}d\sigma_N d\theta}^{\frac{1}{1+\beta}}
		\\
		&\pa{2\int_{\kac\cap A_{\lambda}}\pa{1+v_1^2+v_2^2}^{\frac{1+\beta}{\beta}} F_N\pa{\bm{V}_N}d\sigma^N}^{\frac{\beta}{1+\beta}}.
	\end{split}
\end{equation}
Using the log-power property and the fact that for all $\eta>0$ 
$$\pa{1+v_1^2+v_2^2}^{\eta} \leq 3^{\eta}\pa{1+\abs{v_1}^{2\eta}+\abs{v_2}^{2\eta}}$$ 
we find that
$$\frac{D_{N,1}(F_N)}{N} \leq \lambda^{1-\gamma} \frac{D_{N,\gamma}(F_N)}{N}+\frac{2^{\frac{\beta}{1+\beta}}3^{\frac{k\beta}{1+\beta}}\mathcal{C}_{\beta}}{2} \frac{1}{\lambda^{\frac{k\beta}{1+\beta}-1}} \pa{1+2M_{2k}}^{\frac{\beta}{1+\beta}}.$$
Optimising over $\lambda$ yields
\begin{equation}\nonumber
	\begin{aligned}
		\frac{D_{N,1}(F_N)}{N}  &\leq  \frac{k\beta-\gamma(1+\beta)}{(1+\beta)(1-\gamma)}\pa{\frac{(1+\beta)(1-\gamma)}{k\beta-(1+\beta)}}^{\frac{k\beta-(1+\beta)}{k\beta-\gamma(1+\beta)}}\\
		&\frac{3^{\frac{k\beta(1-\gamma)}{k\beta-\gamma(1+\beta)}}\mathcal{C}_{\beta}^{\frac{(1+\beta)(1-\gamma)}{k\beta-\gamma(1+\beta)}}}{2^{\frac{1-\gamma}{k\beta-\gamma(1+\beta)}}}\pa{1+2M_{2k}}^{\frac{\beta(1-\gamma)}{k\beta-\gamma(1+\beta)}}
		\pa{\frac{D_{N,\gamma}(F_N)}{N}}^{\frac{k\beta-(1+\beta)}{k\beta-\gamma(1+\beta)}},
	\end{aligned}
\end{equation}
which together with the fact that $\DD_{N,1}\pa{F_N} \geq C_1 \HH_N\pa{F_N}$ concludes the proof.
\end{proof}

While we may want to pet ourselves on the back for a job well done, a few questions do require our attention:
\begin{itemize}
\item In order to be able to use inequality \eqref{eq:desired_inequality_thm} to find a functional inequality for the Boltzmann equation by using the mean field method we need to consider functions that are strongly entropic chaotic and have the log-power property. Are there any such functions at all?
\item Much like when we discussed the notion of chaoticity in \S\ref{subsec:kac_MFL}, we wonder if the log-power property, as well as the notion of strongly entropic chaoticity, propagate with the master equation \eqref{eq:master_kac_extended}?
\end{itemize}

While we will shortly show the existence of log-power states, we still don't know the answer to the second question above. Initial attempts to explore the propagation of the log-power property were met with surprising resistance \amit{and the author is still unsure which ingredient is missing or was overlooked}. What we did achieve, however, is still enough for us to be able to fulfil Kac's original vision of ``pushing down'' information from the many particle system to its limit equation.

\subsection{Kac's vision fulfilled}\label{subsec:kac_vision}

As was mentioned in \S\ref{subsec:spectral_gap}, Kac viewed his model as more than just means to justify the Boltzmann equation. He wanted to use it and the mean field limit approach to gain information about the Boltzmann equation - in particular about its convergence to equilibrium. His spectral gap approach, and later on the entropic study of Kac's model, relied on finding functionals that, up to a simple rescaling, converge to zero as time goes to infinity uniformly in $N$. Consequently, the propagation of associated notions under Kac's master equation was essential to this process.

The approach presented in this section, and in particular inequality \eqref{eq:desired_inequality_thm}, gives us the means to \amit{not only transfer rates of convergence, but to ``push down'' a functional inequality, giving us information about the geometry of the problem and the setting.} Indeed, assuming that we found an $f-$strongly entropic family, $\br{F_N}_{N\in\N}$, with respect to the generalised Kac master equation \eqref{eq:master_kac_extended}, i.e. a family such that 
$$\lim_{N\to\infty}\frac{\HH_N\pa{F_N}}{N}=\HH\pa{f|\MM}$$
and
$$\lim_{N\to\infty}\frac{\DD_{N,\gamma}(f)}{N}=\frac{\DD_{\gamma}(f)}{2}$$
with
\begin{equation}\label{eq:def_of_DD_gamma}
\begin{aligned}
	\DD_{\gamma}\pa{f}=	\frac{1}{2\pi}&\int_{\R\times [0,2\pi]}\pa{1+v^2+v_\ast^2}^\gamma\\
	&\psi\pa{f(v)f(v_\ast),f\pa{v\pa{\theta}}f\pa{v_{\ast}\pa{\theta}}}dv dv_\ast d\theta,				
\end{aligned}
\end{equation}
and assuming that $\br{F_N}_{N\in\N}$ is log-power of order $\beta$ than  \eqref{eq:desired_inequality_thm} would immediately imply that
\begin{equation}\label{eq:infered_functional_inequality}
\begin{gathered}
	\DD_{\gamma}(f) \geq 2C_{k,\gamma,\beta}\pa{\HH\pa{f|\MM}}^{1+\frac{(1-\gamma)(1+\beta)}{k\beta-(1+\beta)}}.
\end{gathered}
\end{equation}
If we can find conditions on $f$ and $\br{F_N}_{N\in\N}$ \textit{which propagate with respect to the generalised Boltzmann-Kac equation} \eqref{eq:kac_boltzmann_extended} such that the associated requirement of Theorem \ref{thm:rescaled_inequality} remain true for all time, we would be able to upgrade \eqref{eq:infered_functional_inequality} to a differential inequality that will give us an \textit{explicit} rate of convergence to equilibrium in the Boltzmann-Kac equation.

Can we, then, find $f-$strongly entropic chaotic families that are also log-power of some order? Conditioned tensorisations rear their head again. The strong chaoticity of these families will easily follow from \eqref{eq:entropy_of_conditioned} and a $\gamma-$modified variation of \eqref{eq:production_of_conditioned}\footnote{we just need to add $\pa{1+v_1^2+v_2^2}^\gamma$ inside the integral.} as long as the conditions of theorem \ref{thm:Z_n_approximation_original} hold. The question of the log-power property is answered in the next theorem, taken from \cite{CCE2018}:

\begin{theorem}\label{thm:conditionedlogpower}
Let $f$ be a probability density on $\R$ that satisfies the conditions of Theorem \ref{thm:Z_n_approximation_original}. Assume in addition that $\norm{f}_\infty <\infty$ and that there exist $\beta>0$ and a positive function $\Phi$ such that
$$f(v) \geq e^{-\Phi(v)},$$
$$M_{\Phi,\beta} = \int_{\R}\Phi(v)^{1+\beta}f(v)dv < \infty,$$
and
$$M_{\text{avg},\Phi,\beta} = \int_{\R^2}\pa{\int_{0}^{2\pi}\Phi(v_1(\theta))^{1+\beta}d\theta}f(v_1)f(v_2)dv_1 dv_2 < \infty.$$
Then the conditioned tensorisation of $f$, $\br{F_N}_{N\in\N}$, has the log-power property of order $\beta$ for $N\geq 3$. Moreover, the log power constant, $\mathcal{C}_{\beta}$, can be chosen to be
\begin{equation}\label{eq:C_of_log_power}
	\mathcal{C}_{\beta}= \left[ 2^{1+2\beta}\sqrt{3}\frac{1+\sqrt{2\pi}\sup_{N\in\N}\sup_{u} \abs{\lambda_{N-1}(u)}}{1-\sqrt{2\pi}\sup_{N\in\N} \sup_u\abs{\lambda_N}(u)}\mathfrak{M}_{\epsilon,\phi,\beta} \right]^{\frac{1}{1+\beta}},
\end{equation}
where $\lambda_N$ is the same as in Theorem \ref{thm:Z_n_approximation_original} and 
$$\mathfrak{M}_{\epsilon,\phi,\beta}=2\pa{\sup_{x\geq 1}\frac{\log x}{x^{\epsilon}}}^{1+\beta}\norm{f}^{\epsilon\pa{1+\beta}}_{\infty}+M_{\Phi,\beta}+M_{\text{avg},\Phi,\beta}.$$
\end{theorem}

We can further restrict the conditions on our generator $f$ to ones that \amit{satisfy the requirements of Theorem \ref{thm:conditionedlogpower} \textit{and} propagate with the generalised Boltzmann-Kac equation, resulting in a \textit{time dependent} family of conditioned tensorisation, $\rpa{f(t)^{\otimes N}}_{N\in\N}$, which is $f(t)-$strongly entropic chaotic \textit{and} log-power of an order which is independent of $N$. Consequently, plugging this family in \eqref{eq:desired_inequality_thm} and taking $N$ to infinity} we can obtain an explicit rate of convergence to equilibrium. We conclude this section with the following theorem from \cite{CCE2018}:

\begin{theorem}\label{thm:almost_cerc_kac_boltz}
Let $f$ be a probability density on $\R$ that satisfies the conditions of Theorem \ref{thm:Z_n_approximation_original}. Assume in addition that 
\begin{itemize}
	\item There exists $\beta>0$ and $k>1+1/\beta$ such that
	$$M_{\max\pa{2k,k(1+\beta),4}}(f)=\int_{\R}\abs{v}^{\max\pa{2k,k(1+\beta),4}}f(v)dv <\infty.$$
	\item  $$I(f) = \int_{\R}\frac{\pa{f^\prime(x)}^2}{f(x)}dx < \infty.$$
	\item $f(v) \geq C e^{-\abs{v}^2}$ 	for all $v\in\R$.
\end{itemize}
Then, there exists an explicit constant, $\mathcal{C}$, depending only on $f$, its moments, and $\beta$, such that
\begin{equation}\label{eq:almost_cerc_kac_boltz}
	D_\gamma(f) \geq \mathcal{C} H\pa{f|M}^{1+\frac{(1-\gamma)(1+\beta)}{k\beta-(1+\beta)}}.
\end{equation}
Moreover, if $f(t)$ is the solution to \eqref{eq:kac_boltzmann_extended} then \eqref{eq:almost_cerc_kac_boltz} holds for all $t$ with constants that are  independent of time. 
\end{theorem}

\section{Final Remarks}\label{sec:final}
The study and ideas that govern mean field limits, chaoticity, and Kac's model have a long and illustrious history. In our review we have mostly focused on the \textit{functional} aspect of this study - in particular for Kac's particle system. We have tried to illustrate where issues arise when one studies the explicit convergence to equilibrium of the system - and tried to show that sometimes finding the right inequality and (family of) function(s) can circumvent some of the dynamical issues in the former study. There is still much to explore within and outside the setting of this model and the author expects that the best is yet to come.

%
%
%

\begin{thebibliography}{10}

\bibitem{AMTU2001}
Arnold A., Markowich P., Toscani G. and Unterreiter A. 
\newblock On convex {S}obolev inequalities and the rate of convergence to
equilibrium for {F}okker-{P}lanck type equations.
\newblock {\em Comm. Partial Differential Equations}, 26(1-2):43--100, 2001.

\bibitem{BC1999}
Bobylev A. and Cercignani C.
\newblock On the rate of entropy production for the {B}oltzmann equation.
\newblock {\em J. Statist. Phys.}, 94(3-4):603--618, 1999.

\bibitem{CC1992}
Carlen E. A. and Carvalho M. C.
\newblock Strict entropy production bounds and stability of the rate of
convergence to equilibrium for the {B}oltzmann equation.
\newblock {\em J. Statist. Phys.}, 67(3-4):575--608, 1992.

\bibitem{CC1994}
Carlen E. A. and Carvalho M. C.
\newblock Entropy production estimates for {B}oltzmann equations with
physically realistic collision kernels.
\newblock {\em J. Statist. Phys.}, 74(3-4):743--782, 1994.

\bibitem{CCE2018}
Carlen E. A., Carvalho M. C. and Einav A.
\newblock Entropy production inequalities for the {K}ac walk.
\newblock {\em Kinet. Relat. Models}, 11(2):219--238, 2018.

\bibitem{CCLLV2010}
Carlen E. A., Carvalho M. C., Le Roux J., Loss M. and Villani C.
\newblock Entropy and chaos in the {K}ac model.
\newblock {\em Kinet. Relat. Models}, 3(1):85--122, 2010.

\bibitem{CCL2003}
Carlen E. A., Carvalho M. C.  and Loss M.
\newblock Determination of the spectral gap for {K}ac's master equation and
related stochastic evolution.
\newblock {\em Acta Math.}, 191(1):1--54, 2003.

\bibitem{CCL2020}
Carlen E. A., Carvalho M. C. and Loss M.
\newblock Spectral gaps for reversible {M}arkov processes with chaotic
invariant measures: the {K}ac process with hard sphere collisions in three
dimensions.
\newblock {\em Ann. Probab.}, 48(6):2807--2844, 2020.

\bibitem{CCL2011review}
Carlen E. A., Carvalho M. C. and Loss M.
\newblock Kinetic theory and the {K}ac master equation.
\newblock In {\em Entropy and the quantum {II}}, volume 552 of {\em Contemp.
	Math.}, pages 1--20. Amer. Math. Soc., Providence, RI, 2011.

\bibitem{CE2013Levi}
Carrapatoso K. and Einav A.
\newblock Chaos and entropic chaos in {K}ac's model without high moments.
\newblock {\em Electron. J. Probab.}, 18:no. 78, 38, 2013.

\bibitem{Einav2011Kac}
Einav A.
\newblock On {V}illani's conjecture concerning entropy production for the {K}ac
master equation.
\newblock {\em Kinet. Relat. Models}, 4(2):479--497, 2011.

\bibitem{GSRT2013}
Gallagher I., Saint-Raymond L. and Texier B.
\newblock {\em From {N}ewton to {B}oltzmann: hard spheres and short-range
	potentials}.
\newblock Zurich Lectures in Advanced Mathematics. European Mathematical
Society (EMS), Z\"{u}rich, 2013.

\bibitem{HM2014}
Hauray M. and Mischler S.
\newblock On {K}ac's chaos and related problems.
\newblock {\em J. Funct. Anal.}, 266(10):6055--6157, 2014.

\bibitem{J2001}
Janvresse E.
\newblock Spectral gap for {K}ac's model of {B}oltzmann equation.
\newblock {\em Ann. Probab.}, 29(1):288--304, 2001.

\bibitem{Kac1956}
Kac M.
\newblock Foundations of kinetic theory.
\newblock In {\em Proceedings of the {T}hird {B}erkeley {S}ymposium on
	{M}athematical {S}tatistics and {P}robability, 1954--1955, vol. {III}}, pages
171--197. University of California Press, Berkeley-Los Angeles, Calif., 1956.

\bibitem{McKean1967}
McKean H. P. Jr. 
\newblock An exponential formula for solving {B}oltmann's equation for a
{M}axwellian gas.
\newblock {\em J. Combinatorial Theory}, 2:358--382, 1967.

\bibitem{MM2013}
Mischler S. and Mouhot C.
\newblock Kac's program in kinetic theory.
\newblock {\em Invent. Math.}, 193(1):1--147, 2013.

\bibitem{TV1999}
Toscani G. and Villani C.
\newblock Sharp entropy dissipation bounds and explicit rate of trend to
equilibrium for the spatially homogeneous {B}oltzmann equation.
\newblock {\em Comm. Math. Phys.}, 203(3):667--706, 1999.

\bibitem{V2002review}
Villani C.
\newblock A review of mathematical topics in collisional kinetic theory.
\newblock volume~1 of {\em Handbook of Mathematical Fluid Dynamics}, pages
71--74. North-Holland, 2002.

\bibitem{V2003}
Villani C.
\newblock Cercignani's conjecture is sometimes true and always almost true.
\newblock {\em Comm. Math. Phys.}, 234(3):455--490, 2003.

\end{thebibliography}
\end{document}